\documentclass{article}

\usepackage{amssymb}
\usepackage{amsmath}
\usepackage{amsthm}
\usepackage{bm}
\usepackage{pb-diagram}
\usepackage{subcaption}
\newcommand{\g}{\mathfrak{g}}

\newcommand{\h}{\mathfrak{h}}
\newcommand{\mv}{\mathfrak{v}}
\newcommand{\mz}{\mathfrak{z}}

\theoremstyle{plain}
\newtheorem{theorem}{Theorem}[section]
\newtheorem{proposition}{Proposition}[section]

\newtheorem{lemma}{Lemma}[section]

\theoremstyle{definition}
\newtheorem{definition}{Definition}[section]

\theoremstyle{remark}
\newtheorem{remark}{Remark}[section]

\usepackage{amscd}
\newcommand{\R}{\mathbb{R}}

\newcommand{\Z}{\mathbb{Z}}
\newcommand{\N}{\mathbb{N}}
\newcommand{\T}{\mathbb{T}}
\renewcommand{\H}{\mathbb{H}}

\newcommand{\Ker}{\mathrm{Ker}}
\renewcommand{\Im}{\mathrm{Im}}
\newcommand{\End}{\mathrm{End}}
\newcommand{\Aut}{\mathrm{Aut}}
\renewcommand{\O}{\mathrm{O}}
\newcommand{\GL}{\mathrm{GL}}
\newcommand{\Sp}{\mathrm{Sp}}
\newcommand{\dv}{\mathrm{dvol}}
\newcommand{\Inn}{\mathrm{Inn}}
\newcommand{\Stab}{\mathrm{Stab}}
\newcommand{\Ric}{\mathbf{Ric}}
\usepackage[dvipdfmx]{graphicx}

\begin{document}

\title{Mahler type theorem and non-collapsed limits of sub-Riemannian compact Heisenberg manifolds}
\author{Kenshiro Tashiro}
\date{}
\maketitle
\begin{abstract}
            In this paper,
we study a non-collapsed Gromov--Hausdorff limit of a sequence of compact Heisenberg manifolds with sub-Riemannian metrics.
In the case of strictly sub-Riemannian case, we show that if a sequence has an upper bound of the diameter and a lower bound of Popp's measure, then it has a convergent subsequence in the Gromov--Hausdorff topology, and the limit is isometric to a compact Heisenberg manifold of the same dimension.
The same conclusion is also shown for Riemannian case with the additional assumption on Ricci curvature lower bounds.

\end{abstract}

\section{Introduction}

In Riemannian geometry,
the topological/differential structure of the Gromov--Hausdorff limit of a sequence of Riemannian manifolds $\{M_k\}$ is actively studied under the assumption of the sectional curvature or the Ricci curvature.
The easiest case is flat manifolds,
that is Riemannian manifolds with the sectional curvature is $0$.
Denote by $\mathcal{M}_{flat}(n,D,V)$ the set of $n$-dimensional compact flat Riemannian manifolds with a diameter upper bound $D$ and a volume lower bound $V$.
If a sequence in $\mathcal{M}_{flat}(n,D,V)$ consists of flat tori,
then the limit is again isometric to a flat torus of the same dimension.
It is a consequence of the classical Mahler's compactness theorem \cite{mah}.
If a sequence consists of flat manifolds which are not necessarily tori,
then the limit space may be a flat orbifold \cite{bet}.
Such (pre)compactness theorem is generalized to all compact Riemannian manifolds by Gromov \cite{gro}.
He showed that the set of compact Riemannian manifolds $\mathcal{M}(n,D,V,K)$, where $n,D,V$ are same to the above and $K$ is a lower bound of the Ricci curvature,
is precompact in the set of compact metric spaces with the Gromov--Hausdorff topology.
Such a limit space are called a \textit{Ricci limit space},
which may not be smooth Riemannian manifolds.

We will study an analogy of these results for compact quotients of the Heisenberg Lie group with left invariant Riemannian/sub-Riemannian metrics.
The Heisenberg Lie group $\H_n$ is the simply connected nilpotent Lie group diffeomorphic to $\mathbb{C}^n\times\R$ with the law of group product
$$(\mathbf{w}_1,z_1)\cdot(\mathbf{w}_2,z_2)=(\mathbf{w}_1+\mathbf{w}_2,z_1+z_2+\frac{1}{2}\mathfrak{Im}(\mathbf{w}_1\cdot \mathbf{w}_2)),$$
where $\mathfrak{Im}(\mathbf{w}_1\cdot \mathbf{w}_2)$ is the imaginary part of the hermitian product.
Let $\Gamma$ be a lattice in $\H_n$.
We call the quotient space $\Gamma\backslash \H_n$ a \textit{compact Heisenberg manifold}.
A left invariant Riemannian/sub-Riemannian metric on $\H_n$ is given by data $(\mv,\langle\cdot,\cdot\rangle)$,
where $\mv$ is a subspace of the Lie algebra $\h_n$ and $\langle\cdot,\cdot\rangle$ is a scalar product on $\mv$.
These data yield the length structure and the geodesic distance on $\H_n$, see Section \ref{sec0}.
We say that a sub-Riemannian metric on $\Gamma\backslash \H_n$ is \textit{left invariant} if the pullback metric on the universal cover $\H_n$ is left invariant.
We call $k:=\dim(\h_n/\mv)$ the corank of a sub-Riemannian metric $(\mv,\langle\cdot,\cdot\rangle)$.
Notice that corank $0$ sub-Riemannian metric is Riemannian.

One of the main theorem is non-collapsed limits for Riemannian manifolds.
\begin{theorem}\label{mainthm}
	Let $\{(\Gamma_k\backslash \H_n,\langle\cdot,\cdot\rangle_k)\}_{k\in\N}$ be a sequence of compact Heisenberg manifolds with left invariant Riemannian metrics which converges in the Gromov--Hausdorff topology.
 Assume that there are constants $D,V>0$ and $K\in\R$ such that
 \begin{itemize}
     \item[1] $diam(\Gamma_k\backslash \H_n,\langle\cdot,\cdot\rangle_k)\leq D$,
     \item[2] $meas(\Gamma_k\backslash \H_n,\langle\cdot,\cdot\rangle_k)\geq V$
 \item[3]$Ric_{\langle\cdot,\cdot\rangle_k}\geq K$.
 \end{itemize}
 Then the Gromov--Hausdorff limit is isometric to a compact Heisenberg manifold with left invariant Riemannian metric of the same dimension.
 \end{theorem}
On Riemannian manifolds,
one has a canonical isometry invariant volume form $\dv_R$.
The total measure $meas(\Gamma_k\backslash \H_n,\langle\cdot,\cdot\rangle_k)$ is defined by the integral of this volume form.
Its generalization to sub-Riemannian manifolds is called \textit{Popp's volume} $\dv_{sR}$.
For its definition, see Section \ref{secpopp}.
Another main theorem is on non-collapsed limits of sub-Riemannian compact Heisenberg manifolds.

\begin{theorem}\label{mainthm2}
	Let $\{(\Gamma_k\backslash H_n,\mv_k,\langle\cdot,\cdot\rangle_k)\}_{k\in\N}$ be a sequence of compact Heisenberg manifolds with left invariant sub-Riemannian metrics of corank $1$ which converges in the Gromov--Hausdorff topology.
 Assume that there are constants $D,V>0$ such that
 \begin{itemize}
 \item[1]$diam(\Gamma_k\backslash \H_n,\mv_k,\langle\cdot,\cdot\rangle_k)\leq D$,
 \item[2]$meas(\Gamma_k\backslash \H_n,\mv_k,\langle\cdot,\cdot\rangle_k)\geq V$.
 \end{itemize}
 Then the Gromov--Hausdorff limit is isometric to a compact Heisenberg manifold with left invariant sub-Riemannian metric of the same dimension.
\end{theorem}

The proof is done by the following two steps.
Let us start from the Riemannian case.
First of all,
we show the finiteness theorem for diffeomorphism type of compact Heisenberg manifolds.

\begin{proposition}\label{propfiniteriem}
	Under the assumption in Theorem \ref{mainthm},
	the number of diffeomorphism classes of compact Heisenberg manifolds is finite.
\end{proposition}

By this proposition,
we can assume that a sequence consists of diffeomorphic compact Heisenberg manifolds.
For a fixed diffeomorphism type $\Gamma\backslash \H_n$,
we can parametrize the isometry classes of left invariant Riemannian metrics $\mathcal{M}_0(\Gamma\backslash \H_n)$ (see Theorem \ref{dfn2-3}).
From this parametrization,
we can define a canonical Hausdorff topology $\mathcal{O}_0$.
We call the pair $(mathcal{M}_0(\Gamma\backslash \H_n),\mathcal{O}_0)$ the \textit{moduli space of Riemannian metrics}.
In \cite{bol2},
Boldt characterize precompact subsets in th moduli space $(\mathcal{M}_0(\Gamma\backslash \H_n),\mathcal{O}_0)$ under specific four conditions on metric tensors.
\begin{theorem}[Corollary 3.14 in \cite{bol2}]\label{thmbol1}
	A subset in $(\mathcal{M}_0(\Gamma\backslash \H_n),\mathcal{O}_0)$ is precompact if and only if the conditions (A-1)-(A-4) hold.
\end{theorem}

We will describe the explicit condition (A-1)-(A-4) in Section \ref{secriem}.
The main contribution of this paper is the geometric interpretation of this assumptions.

\begin{proposition}\label{propriem}
	The conditions (A-1)-(A-4) are equivalent to the assumption of Theorem \ref{mainthm}.
\end{proposition}

We can prove the main theorem \ref{mainthm} provided Proposition \ref{propfiniteriem}, \ref{propriem} and Theorem \ref{thmbol1} hold.
\begin{proof}[Proof of Theorem \ref{mainthm}]
    By the assumption and Proposition \ref{propfiniteriem},
    we can assume $\Gamma_k\simeq \Gamma$ for all $k$.
    By Proposition \ref{propriem} and Theorem \ref{thmbol1},
    there is a subsequence which converges in $(\mathcal{M}_0(\Gamma\backslash \H_n),\mathcal{O}_0)$.
    Since both Gromov--Hausdorff topology and the canonical topology $\mathcal{O}_0$ is Hausdorff,
    the limit in $\mathcal{O}_0$ is unique and it coincides with the Gromov--Hausdorff limit.
\end{proof}

We can generalize the previous arguments to the $(\text{corank }1)$ sub-Riemannian setting.
In the same way to Riemannian case,
we can define the \textit{moduli space of sub-Riemannian metrics of corank $1$},
$(\mathcal{M}_1(\Gamma\backslash \H_n),\mathcal{O}_1)$.

\begin{proposition}\label{propfinitesR}
	Under the assumption in Theorem \ref{mainthm2},
	the number of diffeomorphism classes of compact Heisenberg manifolds is finite.
\end{proposition}

\begin{theorem}[Essentially due to Corollary 3.14 in \cite{bol2}]\label{thmbol2}
	A subset in $(\mathcal{M}_1(\Gamma\backslash H_n),\mathcal{O}_1)$ is precompact  if and only if the conditions (A-1)-(A-3) hold.
\end{theorem}

\begin{proposition}\label{propsR}
	The conditions (A-1)-(A-3) are equivalent to the assumption of Theorem \ref{mainthm2}.
\end{proposition}

These propositions show Theorem \ref{mainthm2}.
Construction of the moduli space and the proof of Proposition \ref{propriem} and \ref{propsR} are the core of this paper.

\section*{Acknowledgement}
The author would like to thank to Koji Fujiwara for many helpful comments.
This work was supported by JSPS KAKENHI Grant Number JP20J13261.

\section{Preliminaries from sub-Riemannian Lie group}\label{sec0}
In this section we prepare notation on sub-Riemannian metrics especially on Lie groups.
\subsection{Sub-Riemannian structure}

	Let $G$ be a connected Lie group,
	$\g$ the associated Lie algebra,
	$\mv\subset \g$ a subspace and $\langle\cdot,\cdot\rangle$ a scalar product on $\mv$.
	For $x\in G$,
	denote by $L_x:G\to G$ (resp. $R_x:G\to G$) the left translation (resp. right translation) by $x$.
	Define a sub-Riemannian metric on $G$ by
	$$\mathcal{D}_x=L_{x\ast}\mv,~~~~g_x(u,v)=\langle L_{x\ast}^{-1}u,L_{x\ast}^{-1}v\rangle.$$
	Such a sub-Riemannian metric $(\mathcal{D},g)$ is called \textit{left-invariant}.
	We sometimes write left-invariant sub-Riemannian metric by $(\mv,\langle\cdot,\cdot\rangle)$.
 Moreover, if $\dim(\g/\mv)=k$,
 we say that a sub-Riemannian metric $(\mv,\langle\cdot,\cdot\rangle)$ is corank $k$.
 Notice that if $\mv=\g$ i.e. corank $0$,
 then $(\g,\langle\cdot,\cdot\rangle)$ is a Riemannian metric.
\begin{remark}
    From now on we shall declare the corank of sub-Riemannian metrics.
    If we do not declare the corank,
    then the word "sub-Riemannian metric" cover sub-Riemannian metrics of all corank.
\end{remark}

For simplicity,
we shall consider a Lie group with a left-invariant sub-Riemannian metric $(G,\mv,\langle\cdot,\cdot\rangle)$.
The associated distance function is given as follows.
We say that an absolutely continuous path $c:[0,1]\to G$ is \textit{admissible} if $\dot{c}(t)\in L_{c(t)\ast}\mv$ a.e. $t\in [0,1]$.
We define the length of an admissible path by
$$length(c)=\int^1_0\sqrt{\langle\dot{c}(t),\dot{c}(t)\rangle}dt.$$
For $x,y\in G$,
define the distance function by
$$dist(x,y)=\inf\left\{length(c)~|~c(0)=x,c(1)=y,\text{$c$ is admissible}\right\}.$$
In general not every pair of points in $G$ is joined by an admissible path.
This implies that the value of the function $dist$ may be the infinity.
The following \textit{bracket generating} condition ensures that any two points are joined by an admissible path.
\begin{definition}[Bracket generating distribution]\label{dfnbracket}
	For a sub-Riemannian Lie group $(G,\mv,\langle\cdot,\cdot\rangle)$ and an integer $i\in\N$,
	let $\mv^i$ be the subspace in $\g$ inductively defined by
	$$\mv^1=\mv,~~~~\mv^{i+1}=\mv+[\mv,\mv^i].$$
	We say that a subspace $\mv$ is bracket generating if there is $r\in\N$ such that $\mv^r=\g$.
We call $(G,\mv,\langle\cdot,\cdot\rangle)$ is $r$-step if $\mv^{r-1}\subsetneq\mv^r=\g$. 
\end{definition}

\begin{theorem}[See e.g. Theorem 3.31 in \cite{agr}]
	Let $(G,\mv,\langle\cdot,\cdot\rangle)$ be a sub-Riemannian Lie group with a bracket generating distribution.
	Then the following two assertions hold.
	\begin{itemize}
		\item[1]$(G,dist)$ is a metric space,
		\item[2]the topology induced by $dist$ is equivalent to the manifold topology.
		\end{itemize}
		In particular,
		$dist:G\times G\to \R$ is continuous.
\end{theorem}

\begin{remark}
Since the sub-Riemannian structure is left-invariant,
the distance function is also left-invariant,
that is,
$dist(hx,hy)=dist(x,y)$ for all $h,x,y\in G$.
\end{remark}

\subsection{Popp's volume}\label{secpopp}
On a Riemannian Lie group $(G,g)$,
one has a canonical volume form defined by
$$\dv_R=\nu_1\wedge\cdots\wedge\nu_n,$$
where $\{\nu_1,\cdots\nu_n\}$ is a dual coframe of an orthonormal basis.
The induced measure $meas(\Omega):=\left|\int_{\Omega}\dv_R\right|$ $(\Omega\subset G)$ is called the \textit{volume measure}.

In sub-Riemannian geometry,
we also have a canonical volume form,
called \textit{Popp's volume} introduced in \cite{mon}.
For simplicity,
we only consider the $2$-step case.

We do not introduce the original definition of Popp's volume,
however,
we define it with local coordinates given in \cite{barriz}.
Let $U_1,\dots,U_n$ be an adapted frame.
Define the constant $c_{ij}^l$ by
$$[U_i,U_j]=\sum_{l=1}^nc_{ij}^lU_l.$$
We call them the \textit{structure constants}.
We define the $(n-m)$ square matrix $\mathbf{B}$ by
$$B_{hl}=\sum_{i,j=1}^m c_{ij}^h c_{ij}^l.$$

\begin{theorem}[Theorem 1 in \cite{barriz}]\label{thmpopp}
	Let $U_1,\dots,U_n$ be a local adapted frame,
	and $\nu^1,\dots,\nu^n$ the dual coframe.
	Then Popp's volume $\dv_{sR}$ is locally written by
	$$\dv_{sR}=\left(\det \mathbf{B}\right)^{-\frac{1}{2}}\nu^1\wedge\cdots\wedge \nu^n.$$
\end{theorem}
The induced measure $meas(\Omega):=\left|\int_\Omega\dv_{sR}\right|$ ($\Omega\subset G$) is called Popp's measure.

\begin{remark}
If a sub-Riemannian metric is corank $0$ i.e. $1$-step,
then Popp's volume coincides with the canonical volume form.
Indeed,
adapted frame of corank $0$ sub-Riemannian metric is an orthonormal basis.
\end{remark}

\section{Moduli space of left invariant metrics on compact nilmanifolds}\label{sec1}

Let $\Gamma< \H_n$ be a lattice.
We shall give a parametrization of isometry classes of left invariant sub-Riemannian metrics of corank $k=0,1$ on $\Gamma\backslash \H_n$.
Our argument is based on the construction of the moduli space of Riemannian metrics given by Gordon and Wilson in \cite{gor2}.

For a parametrization $(\mathbf{w},z)=(x_1,\dots,x_n,y_1,\dots,y_n,z)$ of $\H_n\simeq \mathbb{C}^n\times\R$,
we fix the basis $\left\{X_1,\dots,X_n,Y_1,\dots,Y_n,Z\right\}$ of the Lie algebra $\h_n$ by
$$X_i=\partial_{x_i}-\frac{y_i}{2}\partial_z,~~Y_i=\partial_{y_i}+\frac{x_i}{2}\partial_z,~~Z=\partial_z.$$
A straightforward computation shows that $[X_i,Y_{i}]=Z$ for all $i=1,\dots,n$ and the other brackets are zero.

Let $\exp:\h_n\to \H_n$ be the exponential map.
It is well known that the exponential map is a diffeomorphism.
We shall identify the Heisenberg Lie group $\H_n$ to its Lie algebra $\h_n$ via the exponential map.

Let $\langle\cdot,\cdot\rangle_0$ be the scalar product on $\h_n$ whose orthonormal basis is $\{X_1,\dots,Y_n,Z\}$.
For a linear endomorphism $\mathbf{A}\in \End_{\R}(\h_n)\simeq \mathrm{Mat}_{2n+1}(\R)$,
let $\langle\cdot,\cdot\rangle_\mathbf{A}$ be a scalar product on $\Im(\mathbf{A})$ induced from the norm
\[\|U\|_\mathbf{A}:=\min\{\|V\|_0\mid U=\mathbf{A}V\}.\]
It is equivalent to the following definition;
For $U,V\in \Im(\mathbf{A})$,
\[\langle U,V\rangle_\mathbf{A}:=\langle \mathbf{A}|_{\Ker(\mathbf{A})^{\perp}}^{-1}U,\mathbf{A}|_{\Ker(\mathbf{A})^{\perp}}^{-1}V\rangle_0,\]
where $\Ker(\mathbf{A})^{\perp}$ is the complementary subspace of $\Ker(\mathbf{A})$ with respect to the scalar product $\langle\cdot,\cdot\rangle_0$.
A pair $(\Im(\mathbf{A}),\langle\cdot,\cdot\rangle_\mathbf{A})$ gives a $($possibly non bracket generating$)$ sub-Riemannian metric.
Notice that $\Im(\mathbf{A})$ is bracket generating if and only if \begin{equation}\label{bracketgen}
    \Im(\mathbf{A})+\mathrm{Span}\{Z\}=\h_n.
\end{equation}
We shall denote the induced left invariant distance function on $\H_n$ by $dist_\mathbf{A}$.
Since a sub-Riemannian metric $(\Im(\mathbf{A}),\langle\cdot,\cdot\rangle_\mathbf{A})$ is left invariant,
we can define the sub-Riemannian metric on $\Gamma\backslash \H_n$ via the quotient map.
We shall denote such a quotient sub-Riemannian metric on $\Gamma\backslash H_n$ by $\overline{dist}_\mathbf{A}$.

We will classify the quotient sub-Riemannian metrics up to isometry.

\begin{lemma}\label{lemisometricsR}
Let $\mathbf{A},\mathbf{B}\in End_{\R}(\h_n)$ be matrices with the bracket generating condition (\ref{bracketgen}).
	Then $(\Gamma\backslash \H_n,\overline{dist}_\mathbf{A})$ is isometric to $(\Gamma\backslash \H_n,\overline{dist}_\mathbf{B})$ if and only if there is an automorphism $\Phi\in \Inn(\H_n)\cdot \Stab(\Gamma)<\Aut(\H_n)$ and $\mathbf{R}\in \O(2n+1)$ such that
 \[\Phi_\ast \mathbf{A}\mathbf{R}=\mathbf{B}.\]
	
\end{lemma}

	Our proof is based on Theorem 5.4 in \cite{gor},
 where the authors characterize the isometry classes of compact nilmanifolds with left invariant Riemannian metrics.
	For the proof,
 we use the following fact.
\begin{theorem}[Special case of Theorem 2 in \cite{kiv}]\label{thm2-3}
Let $dist_\mathbf{A},dist_\mathbf{B}$ be isometric left invariant sub-Riemannian metrics on $\H_n$.
	Then every isometry from $(\H_n,dist_\mathbf{A})$ to $(\H_n,dist_\mathbf{B})$ is a composition of an automorphism and an left translation.
\end{theorem}

\begin{proof}[Proof of Lemma \ref{lemisometricsR}]
	Let $\overline{F}:(\Gamma\backslash \H_n,\overline{dist}_\mathbf{A})\to (\Gamma\backslash \H_n,\overline{dist}_\mathbf{B})$ be an isometry.
	The isometry $\overline{F}$ lifts to an isometry $F:(\H_n,dist_\mathbf{A})\to(\H_n,dist_\mathbf{B})$.
By Theorem \ref{thm2-3},
there is an automorphism $\Phi_0\in \Aut(\H_n)$ such that $F=L_{F(e)}\circ \Phi_0$.
By left-invariance,
the automorphism $\Phi_0$ is also an isometry $(\H_n,dist_\mathbf{A})\to(\H_n,dist_\mathbf{B})$.
Since a smooth sub-Riemannian isometry preserves the metric tensor,
	we have 
 \[\Phi_{0\ast}(\Im(\mathbf{A}))=\Im(\mathbf{B})~~\text{and}~~\Phi_0^{\ast}\left\langle\cdot,\cdot\right\rangle_\mathbf{B}=\left\langle\cdot,\cdot\right\rangle_\mathbf{A}.\]

	The mapping $F\circ \Phi_0^{-1}$ is a self-isometry of $(\H_n,dist_\mathbf{B})$.
	Choose $x\in \H_n$ so that $\sigma=L_x\circ F\circ \Phi_0^{-1}$ is an isometry of $(\H_n,dist_\mathbf{B})$ preserving the identity.
	Again by Theorem \ref{thm2-3},
	$\sigma\in \Aut(\H_n)$.
Thus we obtain
	\[R_x\circ F=L_x^{-1}\circ R_x\circ \sigma\circ \Phi_0\in \Aut(\H_n).\]
Since the mapping $R_x\circ F$ is the lift of a diffeomorphism $R_x\circ \overline{F}:\Gamma\backslash \H_n\to\Gamma\backslash \H_n$,
$R_x\circ F\in \Stab(\Gamma)$ and $\sigma\circ\Phi_0\in \Inn(\H_n)\cdot \Stab(\Gamma)$.
Moreover since $\sigma$ is a self-isometry of $(\H_n,dist_\mathbf{B})$,
\begin{align*}
	&(\sigma\circ\Phi_0)^{\ast}\langle\cdot,\cdot\rangle_\mathbf{B}=\Phi_0^{\ast}\sigma^{\ast}\langle\cdot,\cdot\rangle_\mathbf{B}=\Phi_0^{\ast}\langle\cdot,\cdot\rangle_\mathbf{B}=\langle\cdot,\cdot\rangle_\mathbf{A},\\
	&(\sigma\circ \Phi_0)_{\ast}\Im(
\mathbf{A}))=\sigma_\ast\Phi_{0\ast}(\Im(\mathbf{A}))=\sigma_\ast(\Im(\mathbf{B}))=\Im(\mathbf{B}).
\end{align*}
Then $\Phi=\sigma\circ\Phi_0$ is the desired automorphism.
Indeed,
for all $U,V\in (\Ker \mathbf{A})^{\perp}$,
\begin{align*}
&\langle U,V\rangle_0=\langle \mathbf{A}U,\mathbf{A}V\rangle_\mathbf{A}=\Phi^\ast\langle \mathbf{A}U,\mathbf{A}V\rangle_\mathbf{B}\\
=&\langle \Phi_\ast \mathbf{A}U,\Phi_\ast \mathbf{A}V \rangle_\mathbf{B}=\langle \mathbf{B}|_{\Ker \mathbf{B}^{\perp}}^{-1}\Phi_\ast \mathbf{A}U,\mathbf{B}|_{\Ker \mathbf{B}^{\perp}}^{-1}\Phi_\ast \mathbf{A}V\rangle_0.
\end{align*}
Hence there is $\mathbf{R}\in \O(2n+1)$ such that $\mathbf{R}=\mathbf{B}^{-1}\Phi_\ast \mathbf{A}$ on $(\Ker \mathbf{A})^{\perp}$ and $\mathbf{R}(\Ker \mathbf{A})=\Ker \mathbf{B}$.

Next we show the converse implication.
Let $\Phi=L_x\circ R_x^{-1}\circ \varphi\in \Inn(H_n)\cdot \Stab(\Gamma)$ be an automorphism of $\H_n$ such that $\Phi_\ast \mathbf{A}\mathbf{R}=\mathbf{B}$,
in other words,
$\Phi_{\ast}(\Im(\mathbf{A}))=\Im(\mathbf{B})$ and $\Phi^{\ast}\langle\cdot,\cdot\rangle_\mathbf{B}=\langle\cdot,\cdot\rangle_\mathbf{A}$.
	Since $(\Im(\mathbf{A}),\langle\cdot,\cdot\rangle_\mathbf{A})$ is left-invariant,
	we have $\Phi_{\ast}(\Im(\mathbf{A}))=(R_x^{-1}\circ \varphi)_{\ast}(\Im(\mathbf{A}))$
	and $\Phi^{\ast}(\langle\cdot,\cdot\rangle_\mathbf{B})=(R_x^{-1}\circ \varphi)^{\ast}(\langle\cdot,\cdot\rangle_\mathbf{A})$.
	Moreover, the quotient of the map $R_x^{-1}\circ\varphi$ to $\Gamma\backslash \H_n$ is a diffeomorphism.
This quotient map is the desired isometry.
\end{proof}

By Lemma \ref{lemisometricsR},
isometry classes of sub-Riemannian metrics are characterized by the doublecoset space of matrices with the left multiplication 
 of $\Inn(\H_n)_\ast\cdot\Stab(\Gamma)_\ast$ and the right multiplication of $\O(2n+1)$.
 Up to the actions by $\Inn(\H_n)$ and $\O(2n+1)$,
 a matrix $\mathbf{A}\in \End_{\R}(\h_n)$ has the following canonical representative,
 which is independent of the choice of lattices.

\begin{lemma}\label{lem3-1}
	For $\mathbf{A}\in \End_{\R}(\h_n)$ with the bracket generating condition (\ref{bracketgen}),
	there is $\mathbf{R}\in \O(2n+1)$ and $\mathbf{P}=\Phi_\ast\in \Inn(\H_n)_{\ast}$ such that
	\begin{equation}\label{eq2-3}
	\mathbf{P}\mathbf{A}\mathbf{R}=\begin{pmatrix}
		\tilde{\mathbf{A}} & O\\
		 O & \rho_\mathbf{A}
	\end{pmatrix},
\end{equation}
	where $\tilde{\mathbf{A}}$ is a $2n\times 2n$ invertible matrix and $\rho_\mathbf{A}\in\R$.

	Moreover,
	let $\mathbf{P}'\in \Inn(H_n)_{\ast}$ and $\mathbf{R}'\in \O(2n+1)$ be other matrices such that (\ref{eq2-3}) hold.
Then they are unique in the following sense.
	\begin{itemize}
		\item There is $\tilde{\mathbf{R}}\in \O(2n)$ such that $\mathbf{R}'=\mathbf{R}\begin{pmatrix}
				\tilde{\mathbf{R}} & O\\
		O & \pm 1
	\end{pmatrix}$,
\item $\mathbf{P}'=\mathbf{P}$.
\end{itemize}
\end{lemma}

We call such a representative a \textit{matrix of weak canonical form}.

\begin{remark}
    If a matrix $\mathbf{A}\in \End_\R(\h_n)$ is invertible,
    then an orthonormal basis of the Riemannian metric $(\Im(\mathbf{A}),\langle\cdot,\cdot\rangle_\mathbf{A})$ is
    \[\{\mathbf{A}X_1,\dots,\mathbf{A}Y_n,\mathbf{A}Z\}=\{\tilde{\mathbf{A}}X_1,\dots,\tilde{\mathbf{A}}Y_n,\rho_\mathbf{A}Z\}.\]

    On the other hand,
    if a matrix of weak canonical form $\mathbf{A}\in \End_\R(\h_n)$ is corank $1$,
    then an orthonormal basis of the sub-Riemannian metric $(\Im(\mathbf{A}),\langle\cdot,\cdot\rangle_\mathbf{A})$ is
    \[\{\mathbf{A}X_1,\dots,\mathbf{A}Y_n\}=\{\tilde{\mathbf{A}}X_1,\dots,\tilde{\mathbf{A}}Y_n\}.\]
\end{remark}

\begin{proof}
	By multiplicating an appropriate orthogonal group $\mathbf{R}\in \O(2n+1)$ from the right,
	we can assume that the matrix $\mathbf{A}\mathbf{R}$ maps $[\h_n,\h_n]$ onto itself.
	Then we can write the matrix $\mathbf{A}\mathbf{R}$ as
	\begin{equation}\label{eq2-4}
		\mathbf{A}\mathbf{R}=\begin{pmatrix}
		\tilde{\mathbf{A}} & O\\
		\mathbf{a} & \rho_\mathbf{A}
	\end{pmatrix}\end{equation}
	with an invertible matrix $\tilde{\mathbf{A}}\in \GL_{2n}(\R)$,
	a vector $\mathbf{a}\in\R^{2n}$ and $\rho_\mathbf{A}\in\R$.
	
	For $g=(x_1,\dots,x_n,y_1,\dots,y_n,z)\in \H_n$,
	the matrix representation of the differential of the inner automorphism $\mathbf{P}_g=(i_g)_{\ast}$ is written by
	$$\mathbf{P}_g=\begin{pmatrix}
		I_{2n} & O\\
		\tilde{\mathbf{g}} & 1
	\end{pmatrix},$$
	where $\tilde{\mathbf{g}}=\left(-y_1,\dots,-y_{n},x_1,\dots,x_n\right)$ and $I_{2n}$ is the identity matrix of rank $2n$.
	With this terminology,
	we can write the matrix $\mathbf{P}_g\mathbf{A}\mathbf{R}$ as
	\begin{equation}\label{innerprod}
		\mathbf{P}_g\mathbf{A}\mathbf{R}=\begin{pmatrix}
		\tilde{\mathbf{A}} & O\\
		\mathbf{a}+\tilde{\mathbf{g}}\tilde{\mathbf{A}} & \rho_\mathbf{A}
	\end{pmatrix}.\end{equation}
	Since $\tilde{\mathbf{A}}$ is invertible,
	we can take a unique $\tilde{\mathbf{g}}$ such that $\mathbf{a}+\tilde{\mathbf{g}}\tilde{\mathbf{A}}=0$.

	\vspace{10pt}

	Next we will prove the uniqueness of such matrices.
	Let $\mathbf{R}'\in \O(2n+1)$ be another orthogonal matrix such that
	\begin{equation}\label{eqrprime}
		\mathbf{A}\mathbf{R}'=\begin{pmatrix}
		\tilde{\mathbf{A}}^{\prime} & O\\
		\mathbf{a}^{\prime} & \rho_\mathbf{A}^{\prime}
	\end{pmatrix}\end{equation}
 holds with some $\tilde{\mathbf{A}}^{\prime}\in\GL_{2n}(\R),\rho_\mathbf{A}^{\prime}\in\R$ and $\mathbf{a}^{\prime}\in\R^{2n}$.
	Let us write matrices $\mathbf{A}$,
	$\mathbf{R}$ and $\mathbf{R}'$ as collections of vectors by
	$$\mathbf{A}=\begin{pmatrix}
		\mathbf{a}_1\\
		\vdots\\
		\mathbf{a}_{2n+1}
	\end{pmatrix},~
	\mathbf{R}=\begin{pmatrix}
		\mathbf{r}_1 & \cdots & \mathbf{r}_{2n+1}
		\end{pmatrix},~
		\mathbf{R}'=\begin{pmatrix}
		\mathbf{r^{\prime}}_1 & \cdots & \mathbf{r^{\prime}}_{2n+1}
	\end{pmatrix}.
	$$
From the equality (\ref{eq2-4}) and (\ref{eqrprime}),
we have
$$\begin{cases}
	\mathbf{a}_i\cdot \mathbf{r}_{2n+1}=\mathbf{a}_i\cdot \mathbf{r^{\prime}}_{2n+1}=0~~~~\text{for}~1\leq i\leq 2n,\\
	\mathbf{a}_{2n+1}\cdot \mathbf{r}_{2n+1}=\rho_\mathbf{A},\\
	\mathbf{a}_{2n+1}\cdot \mathbf{r^{\prime}}_{2n+1}=\rho_\mathbf{A}^{\prime}
\end{cases}.$$
The first equality implies that the two vectors $\mathbf{r}_{2n+1}$ and $\mathbf{r^{\prime}}_{2n+1}$ are unit vector orthogonal to $\textrm{Span}\{\mathbf{a}_1,\dots,\mathbf{a}_{2n}\}$.
Hence we have $\mathbf{r}_{2n+1}=\pm \mathbf{r^{\prime}}_{2n+1}$ and $\rho_\mathbf{A}=\pm \rho_\mathbf{A}^{\prime}$.

Moreover,
since $\left\{\mathbf{r}_1,\dots,\mathbf{r}_{2n}\right\}$ and $\left\{\mathbf{r^{\prime}}_1,\dots,\mathbf{r^{\prime}}_{2n}\right\}$ are orthonormal bases of the orthogonal complement of $\mathbf{r}_{2n+1}$,
there is $\tilde{\mathbf{R}}\in O(2n)$ such that
$$^{t}\mathbf{R}\mathbf{R}'=\begin{pmatrix}
	\tilde{\mathbf{R}} & O\\
	O & \pm 1
\end{pmatrix}.$$
This shows the uniqueness of orthogonal matrices $\mathbf{R}$.

We pass to the uniqueness of the inner automorphism $\mathbf{P}_g$.
Let $g'$ be another element in $\H_n$ and $\mathbf{R}'$ the orthogonal matrix such that (\ref{eqrprime}) holds.
In the same way to (\ref{innerprod}),
the matrix representation of $\mathbf{P}_{g^{\prime}}\mathbf{A}\mathbf{R}'$ is
$$\mathbf{P}_{g^{\prime}}\mathbf{A}\mathbf{R}'=\begin{pmatrix}
	\tilde{\mathbf{A}}\tilde{\mathbf{R}} & O\\
	\left(\mathbf{a}+\tilde{\mathbf{g}^{\prime}}\tilde{\mathbf{A}}\right)\tilde{\mathbf{R}} & \pm \rho_\mathbf{A}
\end{pmatrix}.$$
Since $\tilde{\mathbf{A}}$ and $\tilde{\mathbf{R}}$ are invertible,
the matrix $\mathbf{P}_{g^{\prime}}\mathbf{A}\mathbf{R}'$ is weak canonical if and only if $\tilde{\mathbf{g}}=\tilde{\mathbf{g}}^{\prime}$.

\end{proof}

Next we compute the stabilizer group $\Stab(\Gamma)$.
We recall the classification of isomorphism classes of lattices.
Let $D_n$ be the set of $n$-tuples of positive integers $\mathbf{r}=(r_1,\dots,r_n)$ such that $r_{i}$ divides $r_{i+1}$ for all $i=1,\dots,n$.
For $\mathbf{r}\in D_n$,
let $\Gamma_{\mathbf{r}}< \H_n$ be a subgroup generated by $r_1X_1,\dots,r_nX_n,Y_1,\dots,Y_n,Z$, that is
$$\Gamma_{\mathbf{r}}=\langle r_1X_1,\dots,r_nX_n,Y_1,\dots,Y_n,Z\rangle.$$

\begin{theorem}[Theorem 2.4 in \cite{gor2}]
	Any lattice $\Gamma<\H_n$ is isomorphic to $\Gamma_{\mathbf{r}}$ for some $\mathbf{r}\in D_n$.
Moreover,
$\Gamma_{\mathbf{r}}$ is isomorphic to $\Gamma_{\bm{s}}$ if and only if $\mathbf{r}=\bm{s}$.
\end{theorem}
We compute $\Stab(\Gamma_{\mathbf{r}})_{\ast}$ for a fixed $n$-tuple $\mathbf{r}\in D_n$.
Let
$$\mathbf{diag}(\mathbf{r})=\mathbf{diag}(r_1,\dots,r_n,1,\dots,1)$$
be the diagonal $2n\times 2n$ matrix,
and define the skew-symmetric matrix
$$\mathbf{J}_n=\begin{pmatrix}
	O & \mathbf{I}_n\\
	-\mathbf{I}_n & O
\end{pmatrix}.$$
Let $\widetilde{\Sp}(2n,\R)$ be the union of symplectic and anti-symplectic matrices
$$\widetilde{\Sp}(2n,\R)=\left\{\mathbf{S}\in \GL_{2n}(\R)\mid ^{t}\mathbf{S} \mathbf{J}_n
	\mathbf{S}=\epsilon(\mathbf{S})\mathbf{J}_n
,~\epsilon(\mathbf{S})=\pm 1 \right\}.$$
We embed $\widetilde{\Sp}(2n,\R)$ into $\GL_{2n+1}(\R)$ via the mapping $\iota:\mathbf{S}\mapsto \begin{pmatrix}
	\mathbf{S} & 0\\
	0 & \epsilon(\mathbf{S})
\end{pmatrix}$.
With these notations,
we give representations of $\Stab(\Gamma_{\mathbf{r}})_{\ast}$ as follows.

\begin{theorem}[Theorem 2.7 in \cite{gor2}]
	The matrix representation of $\Stab(\Gamma_{\mathbf{r}})_{\ast}$ is given by
	$$\Pi_{\mathbf{r}}=\iota\left(G_{\mathbf{r}}\cap \widetilde{\Sp}(2n,\R)\right),$$
	where $G_{\mathbf{r}}=\mathbf{diag}(\mathbf{r}) \GL_{2n}(\Z)\mathbf{diag}(\mathbf{r})^{-1}$.
\end{theorem}

Now we have prepared to give a parametrization of sub-Riemannian metrics of corank $k=0,1$.

\begin{theorem}\label{dfn2-3}
	The isometry classes of sub-Riemannian metrics of corank $k=0$ or $1$ on $\Gamma_{\mathbf{r}}\backslash \H_n$ has one-to-one correspondance with the double coset space
	\[
		\mathcal{M}(\Gamma_{\mathbf{r}}\backslash \H_n):=\Pi_{\mathbf{r}}\backslash \left(\GL_{2n}(\R)\times \R\right)/\left(\O(2n)\times\left\{\pm 1\right\}\right),
\]
where Riemannian metrics correspond to
\[\mathcal{M}_0(\Gamma_{\mathbf{r}}\backslash \H_n):=\Pi_{\mathbf{r}}\backslash \left(\GL_{2n}(\R)\times \R^*\right)/\left(\O(2n)\times\left\{\pm 1\right\}\right),\]
and sub-Riemannian metrics of corank $1$ correspond to
\[\mathcal{M}_1(\Gamma_{\mathbf{r}}\backslash \H_n)=\Pi_{\mathbf{r}}\backslash \left(\GL_{2n}(\R)\times \{0\}\right)/\left(\O(2n)\times\left\{\pm 1\right\}\right).\]

\end{theorem}

For $k=0,1$,
denote by $\mathcal{O}_k$ the quotient topology on $\mathcal{M}_k(\Gamma_{\mathbf{r}}\backslash \H_n)$.
We call the topological space $(\mathcal{M}_k(\Gamma_{\mathbf{r}}\backslash \H_n),\mathcal{O}_k)$ the \textit{moduli space of corank $k$}.

Notice that the action on the second factor of the ambient space $\GL_{2n}(\R)\times \R$ only changes its sign.
Moreover,
any matrices acting upon the first factor has determinant $\pm 1$.
This observation shows the following lemma.
\begin{lemma}\label{lem2-2}
	For matrices $\mathbf{A},\mathbf{B}\in \End(\h_n)$ of weak canonical form with $[\mathbf{A}]=[\mathbf{B}]$ in $\mathcal{M}(\Gamma_{\mathbf{r}}\backslash \H_n)$,
	we have
	\begin{itemize}
		\item $|\det(\tilde{\mathbf{A}})|=|\det(\tilde{\mathbf{B}})|$,
		\item $|\rho_\mathbf{A}|=|\rho_\mathbf{B}|.$
		\end{itemize}
\end{lemma}

\begin{remark}\label{rmk2-1}
	The double coset space $G_{\mathbf{r}}\backslash GL_{2n}(\R)/O(2n)$ is homeomorphic to $GL_{2n}(\Z)\backslash GL_{2n}(\R)/O(2n)$,
	which is the moduli space of flat tori of dimension $2n$.
	However,
	the discrete subgroup $G_{\mathbf{r}}\cap \widetilde{Sp}(2n,\R)<G_{\mathbf{r}}\simeq\GL_{2n}(\R)$ is infinite index for $n\geq 2$.
	Hence the classical Mahler's compactness theorem does not directly imply the precompactness of subsets in $\mathcal{M}(\Gamma_{\mathbf{r}}\backslash\H_n)$.
	Boldt's work in \cite{bol2} is to fulfill this gap by using the $\mathrm{j}$-operator defined in the next section.
\end{remark}

\begin{remark}
	The cometric tensor associated to a matrix $\mathbf{A}$ is $\mathbf{A}^{t}\mathbf{A}$ in the basis $\{X_1,\dots,X_n\}$.
\end{remark}

\section{Curvature and volume form on the Heisenberg Lie group}\label{sec3}

In this section,
we recall the Ricci curvature and Popp's volume on the Heisenberg Lie groups.

\subsection{$\mathrm{j}$-operator}

Let $\mathfrak{v}\subset\h_n$ be the subspace spanned by the $2n$ vectors $X_1,\dots,X_n,Y_1,\dots,Y_n$,
and $\omega$ the contact form on $\H_n$ such that $\Ker \omega=\mv$ and $\omega(Z)=1$.
In other words,
the contact form $\omega$ is the dual covector of $Z$.
The following $\mathrm{j}$-operator plays a central role in the study of the Heisenberg Lie group.

\begin{definition}[$\mathrm{j}$-operator]
	For a matrix $\mathbf{A}\in \End(\h_n)$ of weak canonical form,
	define the operator $j_\mathbf{A}:\mv\to \mv$ by
	$$\langle j_\mathbf{A}(U),V\rangle_\mathbf{A}=-d\omega(U,V)=\omega([U,V])~~~~\text{for all}~~U,V\in \mv.$$
\end{definition}

The $\mathrm{j}$-operator was used to describe the curvature tensor (\cite{ebe}),
the cut time of geodesics (\cite{gorn} and \cite{riz}),
and precompact subsets in the moduli space of Riemannian metrics (\cite{bol2}).

The following matrix representation of $j_\mathbf{A}$ is useful for later arguments.

\begin{lemma}\label{lemmat}

	The matrix representation of $j_\mathbf{A}$ in the basis $\left\{\mathbf{A}X_1,\dots,\mathbf{A}X_{2n}\right\}$ is $\tilde{\mathbf{A}}^{t}\mathbf{J}_n\tilde{\mathbf{A}}$ ,
\end{lemma}

\begin{proof}
	For $i,j=1,\dots,2n$,
	the $(i,j)$-th element of the matrix $^{t}\tilde{\mathbf{A}}\mathbf{J}_n\tilde{\mathbf{A}}$ is
	\begin{align*}
		\begin{pmatrix} a_{i\hspace{.1em}1} & \cdots & a_{i\hspace{.1em}2n}\end{pmatrix}
		\mathbf{J}_n
		\begin{pmatrix}
			a_{1\hspace{.1em}j}\\
\vdots\\
a_{2n\hspace{.1em}j}
\end{pmatrix}=\sum_{k=1}^na_{k\hspace{.1em}i}a_{k+n\hspace{.1em}j}-a_{k+n\hspace{.1em}i}a_{k\hspace{.1em}j}.
\end{align*}

It coincides with the $(i,j)$-th element of the matrix representation of $j_\mathbf{A}$.
Indeed,
for example $i,j=1,\dots,n$,
	\begin{align*}
		\langle j_\mathbf{A}(\mathbf{A}X_i),\mathbf{A}X_j\rangle_\mathbf{A}&=\omega([\mathbf{A}X_i,\mathbf{A}X_j])\\
					      &=\omega\left(\left[\sum_{k=1}^{n}a_{k\hspace{.1em}i}X_k+a_{k+n\hspace{.1em}i}Y_k,\sum_{k=1}^{n}a_{k\hspace{.1em}j}X_k+a_{k+n\hspace{.1em}j}Y_k\right]\right)\\
					      &=\sum_{k=1}^{n}a_{k\hspace{.1em}i}a_{k+n\hspace{.1em}j}-a_{k+n\hspace{.1em}i}a_{k\hspace{.1em}j}.
	\end{align*}
 The same computation holds also for $n<i,j<2n$.
\end{proof}

Since the operator $j_\mathbf{A}$ is skew symmetrizable,
its eigenvalue is purely imaginary.

\begin{definition}
	For a matrix $\mathbf{A}\in\End_\R(\h_n)$,
	define the positive numbers $d_1(\mathbf{A})\leq\cdots\leq d_n(\mathbf{A})$ so that $\pm\sqrt{-1}d_1(\mathbf{A}),\dots,\pm\sqrt{-1}d_n(\mathbf{A})$ are the eigenvalues of the $j_\mathbf{A}$.
\end{definition}

\begin{lemma}\label{lemdn}
	For matrices $\mathbf{A},\mathbf{B}\in\End_\R(\h_n)$ of weak canonical form with $[\mathbf{A}]=[\mathbf{B}]$ in $\mathcal{M}(\Gamma\backslash H_n)$,
	we have $d_i(\mathbf{A})=d_i(\mathbf{B})$ for all $i=1,\dots,n$.
\end{lemma}
\begin{proof}
	From the construction of the moduli space,
	there is a matrix $\mathbf{P}\in G_{\mathbf{r}}\cap\widetilde{\Sp}(2n,\R)$ and $\mathbf{R}\in \O(2n)$ such that $\mathbf{P}\tilde{\mathbf{A}}\mathbf{R}=\tilde{\mathbf{B}}$.
	Then we have
	\begin{align*}^{t}\tilde{\mathbf{B}}\mathbf{J}_n\tilde{\mathbf{B}}&= ^{t}\mathbf{R}^{t}\tilde{\mathbf{A}}^{t}\mathbf{P}\mathbf{J}_n\mathbf{P}\tilde{\mathbf{A}}\mathbf{R}=\pm ^{t}\mathbf{R}^{t}\tilde{\mathbf{A}}\mathbf{J}_n\tilde{\mathbf{A}}\mathbf{R}.\end{align*}
		This implies the coincidence of the eigenvalues.
\end{proof}

Moreover, by the skew-symmetricity, we can choose the representatives in the moduli space $\mathcal{M}(\Gamma_{\mathbf{r}}\backslash H_n)$ so that the matrix representation of the $j_\mathbf{A}$ is a block matrix which consists of the eigenvalues of $j_\mathbf{A}$.
\begin{definition}
	We say that a matrix $\mathbf{A}\in\End_\R(\h_n)$ is of canonical form if it is weak canonical and its $\mathrm{j}$-operator has the matrix representation
	$$^{t}\tilde{\mathbf{A}}\mathbf{J}_n\tilde{\mathbf{A}}=\begin{pmatrix}
		O & \mathbf{diag}(d_1(\mathbf{A}),\dots,d_n(\mathbf{A}))\\
		-\mathbf{diag}(d_1(\mathbf{A}),\dots,d_n(\mathbf{A})) & O
	\end{pmatrix}.$$
\end{definition}

\begin{remark}
	The number $d_i(\mathbf{A})$ appeared in Boldt's paper \cite{bol2} as the eigenvalue of the matrix $\tilde{\mathbf{A}}^{t}\tilde{\mathbf{A}}\mathbf{J}_n$.
	This matrix is similar to our matrix representation $\tilde{\mathbf{A}}\mathbf{J}_n\tilde{\mathbf{A}}$,
	hence his definition coincides with ours.
\end{remark}

We introduce another numerical data $\delta(\mathbf{A})$ to give an explicit form of Popp's volume form.
\begin{definition}
	For a matrix $\mathbf{A}\in\End_\R(\h_n)$,
	define $\delta(\mathbf{A}):=\|^{t}\tilde{\mathbf{A}}\mathbf{J}_n\tilde{\mathbf{A}}\|_{HS}$,
 where $\|\cdot\|_{HS}$ is the Hilbert--Schmidt norm.
\end{definition}

\begin{lemma}\label{lemdelta}
	For two matrices $\mathbf{A},\mathbf{B}\in \End_\R(\h_n)$ of canonical form with $[\mathbf{A}]=[\mathbf{B}]$,
	we have $\delta(\mathbf{A})=\delta(\mathbf{B})$.
\end{lemma}

\begin{proof}

From the assumption,
there are $\mathbf{P}\in G_{\mathbf{r}}\cap\widetilde{\Sp}(2n,\R)$ and $\mathbf{R}\in \O(2n)$ such that $\mathbf{P}\tilde{\mathbf{A}}\mathbf{R}=\tilde{\mathbf{B}}$.
Then we have
\begin{align*}
	\delta(\mathbf{B})&=\|^{t}\mathbf{R} ^{t}\tilde{\mathbf{A}} ^{t}\mathbf{P}\mathbf{J}_n\mathbf{P}\tilde{\mathbf{A}}\mathbf{R}\|_{HS}\\
		 &=\|\pm ^{t}\tilde{\mathbf{A}}\mathbf{J}_n\tilde{\mathbf{A}}\|_{HS}\\
		 &=\delta(\mathbf{A}).
\end{align*}

\end{proof}

\subsection{Ricci curvature}

We recall an explicit formula of the Ricci curvature of the Heisenberg Lie group.
This computation is due to Eberlein \cite{ebe}.

For an invertible matrix $\mathbf{A}\in\End_{R}(\h_n)$ of canonical form,
let $\Ric_\mathbf{A}:\h_n\otimes\h_n\to \R$ be the Ricci tensor associated to the left invariant Riemannian metric $\langle\cdot,\cdot\rangle_\mathbf{A}$.

\begin{proposition}[Special case of Proposition 2.5 in \cite{ebe}]\label{propric}
The Ricci tensor associated to the metric $\left\langle\cdot,\cdot\right\rangle_\mathbf{A}$ is written by;

	\begin{itemize}
		\item[1]$\Ric_\mathbf{A}(U,V)=0$ for all $
U\in \mv$ and $V\in [\h_n,\h_n]$,
		\item[2]$\Ric_\mathbf{A}(U,V)=\langle\frac{1}{2\rho_\mathbf{A}^2}j_\mathbf{A}^2U,V\rangle_\mathbf{A}$ for all $U,V\in \mv$,
		\item[3]$\Ric_\mathbf{A}(U,V)=-\frac{\|U\|_\mathbf{A}\|V\|_\mathbf{A}}{4\rho_\mathbf{A}^2}Tr(j_\mathbf{A}^2)$ for all $U,V\in[\h_n,\h_n]$.
	\end{itemize}
\end{proposition}

For a unit vector $U\in (\h_n,\langle\cdot,\cdot\rangle_\mathbf{A})$,
we write $\Ric_\mathbf{A}(U)=\Ric_\mathbf{A}(U,U)$.
The function $\Ric_\mathbf{A}$ is called the \textit{Ricci curvature}.
Combined with the matrix representation of the operator $j_\mathbf{A}$,
we obtain the following lemma.
\begin{lemma}\label{lemric}
	For an invertible matrix $\mathbf{A}\in\End_\R(\h_n)$ of canonical form,
	we have
	$$\begin{cases}
		\Ric_\mathbf{A}(\mathbf{A}X_i)=-\frac{1}{2\rho_\mathbf{A}^2}d_i(\mathbf{A})^2 & \text{for all}~~i=1,\dots,2n\\
		\Ric_\mathbf{A}(\mathbf{A}Z)=\frac{1}{2\rho_\mathbf{A}^2}\sum_{k=1}^nd_k(\mathbf{A})^2.
	\end{cases}$$
\end{lemma}

\begin{proof}
	By Lemma \ref{lemmat},
	the matrix representation of the operator $j_\mathbf{A}^2$ is
	$$\mathbf{diag}(-d_1(\mathbf{A})^2,-d_2(\mathbf{A})^2,\dots,-d_n(\mathbf{A})^2,-d_1(\mathbf{A})^2,\dots,-d_n(\mathbf{A})^2)$$
	in the basis $\left\{\mathbf{A}X_1,\dots,\mathbf{A}X_{2n}\right\}$.
	Hence Proposition \ref{propric} shows that
	\begin{align*}\Ric_\mathbf{A}(\mathbf{A}X_i)=\left\langle \frac{1}{2\rho_\mathbf{A}^2}j_\mathbf{A}^2 \mathbf{A}X_i,\mathbf{A}X_i\right\rangle_\mathbf{A}&=\left\langle-\frac{1}{2\rho_\mathbf{A}^2}\sum_{k=1}^nd_k(\mathbf{A})^2(\mathbf{A}X_k+\mathbf{A}Y_k),\mathbf{A}X_i\right\rangle_\mathbf{A}\\
	&=-\frac{1}{2\rho_\mathbf{A}^2}d_i(\mathbf{A})^2.\end{align*}
for $i=1,\dots,2n$.
	We can calculate $\Ric_\mathbf{A}(\mathbf{A}Z)$ by a straightforward adoption of Proposition \ref{propric}.
\end{proof}

\subsection{Volume forms}
In this section,
we recall an explicit formula of the canonical Riemannian volume form and Popp's volume form on the Heisenberg Lie group.

\subsubsection{Riemannian volume form}
Let $\mathbf{A}\in\End_\R(\h_n)$ be an invertible matrix of canonical form.
Denote by $\dv_R(\mathbf{A})$ the canonical Riemannian volume form.
Since the Riemannian volume form is the wedge of the dual coframe of an orthonormal frame,
we have
\begin{align*}
	\dv_R(\mathbf{A})&=(\mathbf{A}X_1)^{\ast}\wedge\cdots\wedge (\mathbf{A}X_{2n})^{\ast}\wedge (\mathbf{A}Z)^{\ast}\\
	      &=\det(\mathbf{A}^{-1})X_1^{\ast}\wedge\cdots \wedge X_{2n}^{\ast}\wedge Z^{\ast}\\
	      &=(\det (\tilde{\mathbf{A}}))^{-1}\rho_\mathbf{A}^{-1}X_1^{\ast}\wedge\cdots\wedge X_{2n}^{\ast}\wedge Z^{\ast}.
\end{align*}

In particular,
the total measure of a compact Riemannian Heisenberg manifold $(\Gamma_{\mathbf{r}}\backslash \H_n,\langle\cdot,\cdot\rangle_\mathbf{A})$ is
\begin{equation}\label{measureR}
    meas(\Gamma_{\mathbf{r}}\backslash \H_n,\mathbf{A}):=\left|\int_{\Gamma_{\mathbf{r}}\backslash \H_n}\dv_{R}(\mathbf{A})\right|=\prod_{i=1}^nr_i|\det(\tilde{\mathbf{A}})|^{-1}|\rho_\mathbf{A}|^{-1}.
\end{equation}
\subsubsection{Popp's volume form}
Let $\mathbf{A}\in\End_\R(\h_n)$ be a rank $2n$-matrix of canonical form. 
Then $\mathbf{A}$ has a matrix representation
	$\mathbf{A}=\begin{pmatrix}
		\tilde{\mathbf{A}} & \mathbf{0}\\
		\mathbf{0} & 0
\end{pmatrix}$.
Denote by $\dv_{sR}(\mathbf{A})$ Popp's volume associated to the sub-Riemannian metric $(\Im(\mathbf{A}),\mathbf{A})=(\mv,\mathbf{A})$.

Notice that a basis $\{\mathbf{A}X_1,\dots,\mathbf{A}X_{2n},Z\}$ is adapted.
Since $^{t}\tilde{\mathbf{A}}\mathbf{J}_n\tilde{\mathbf{A}}$ is the matrix representation of $j_\mathbf{A}$ in the basis $\left\{\mathbf{A}X_1,\dots,\mathbf{A}X_{2n}\right\}$,
its $(i,j)$-th entry coincides with the structure constant $c_{ij}=\omega([\mathbf{A}X_i,\mathbf{A}X_j])$.
By Theorem \ref{thmpopp},
Popp's volume $\dv_{sR}(\mathbf{A})$ is written by
\begin{align*}
	\dv_{sR}(\mathbf{A})&=\det(\delta(\mathbf{A}))^{-1}(\mathbf{A}X_1)^{\ast}\wedge\cdots\wedge (\mathbf{A}X_{2n})^{\ast}\wedge Z^{\ast}\\
		 &=\delta(\mathbf{A})^{-1}(\det(\tilde{\mathbf{A}}))^{-1}X_1^{\ast}\wedge\cdots\wedge X_{2n}^{\ast}\wedge Z^{\ast}.\label{eqpoppvol}
\end{align*}
In particular,
the total measure of a compact sub-Riemannian Heisenberg manifold $(\Gamma_{\mathbf{r}}\backslash \H_n,\mv,\langle\cdot,\cdot\rangle_\mathbf{A})$ is
\begin{equation}\label{measuresR}
    meas(\Gamma_{\mathbf{r}}\backslash \H_n,\mathbf{A}):=\left|\int_{\Gamma_{\mathbf{r}}\backslash \H_n}\dv_{sR}(\mathbf{A})\right|=\prod_{i=1}^nr_i|\det(\tilde{\mathbf{A}})|^{-1}\delta(\mathbf{A})^{-1}.
\end{equation}

	\subsection{Precompact subset in moduli space}\label{secriem}

	We have prepared to recall Boldt's condition for subsets in Riemannian moduli space $(\mathcal{M}_0(\Gamma_{\mathbf{r}}\backslash \H_n),\mathcal{O}_0)$ being precompact.
 Furthermore,
 based on his idea,
 we also give the condition for subsets in sub-Riemannian moduli space $(\mathcal{M}_1(\Gamma_{\mathbf{r}}\backslash \H_n),\mathcal{O}_1)$ being precompact.
	
Let $\Pr:\H_n\simeq \h_n\to \mv$ be the projection.
Then $\mz_{\mathbf{r}}:=\Pr(\Gamma_{\mathbf{r}})$ is a lattice in $\mv$.
	A compact Heisenberg manifold $\Gamma_{\mathbf{r}}\backslash \H_n$ has the canonical submersion $\Gamma_{\mathbf{r}}\backslash \H_n\to \mz_{\mathbf{r}}\backslash \mv\simeq \T^{2n}$.

	\begin{theorem}[Corollary 3.14 in \cite{bol2}]\label{thmboltext}
		A subset $\mathcal{X}\subset (\mathcal{M}_0(\Gamma_{\mathbf{r}}\backslash H_n),\mathcal{O}_0)$ is precompact if and only if there are positive constants $C_1,C_2,C_3,C_-,C_+>0$ such that all $[\mathbf{A}]\in\mathcal{X}$ satisfies the following four conditions;
		\begin{description}
			\item[(A-1)]$\min\left\{\|X\|_\mathbf{A}\mid X\in \mz_{\mathbf{r}}\right\}\geq C_1$,
			\item[(A-2)]$|\det(\tilde{\mathbf{A}})|\geq C_2$,
			\item[(A-3)]$d_n(\mathbf{A})\leq C_3$,
			\item[(A-4)]$C_-\leq |\rho_\mathbf{A}|\leq C_+$.
		\end{description}
	\end{theorem}

	\begin{remark}
		For an invertible matrix $\mathbf{A}$ of canonical form,
		the associated metric tensor is $^{t}\mathbf{A}^{-1}\mathbf{A}^{-1}=\begin{pmatrix}
			^{t}\tilde{\mathbf{A}}^{-1} \tilde{\mathbf{A}}^{-1} & O\\
			O & \rho_\mathbf{A}^{-2}
		\end{pmatrix}$.
		The conditions (A-1)-(A-4) is obtained from Boldt's original statement by adapting this matrix presentation.
	\end{remark}

 The essense of this theorem is the condition (A-1)-(A-3).
 Boldt showed that the condition (A-1)-(A-3) is equivalent to the precompactness in the first factor of the moduli space $\mathcal{M}_0(\Gamma_{\mathbf{r}}\backslash \H_n)$ (Main Theorem in \cite{bol2}).
 Theorem \ref{thmboltext} is a consequence of this fact.
 By the construction of sub-Riemannian moduli space $\mathcal{M}_1(\Gamma_{\mathbf{r}}\backslash \H_n)$,
 we obtain the following theorem immediately.

\begin{theorem}[Essentially by Main Theorem in \cite{bol2}]\label{thmbol2text}
		A subset $\mathcal{X}\subset (\mathcal{M}_1(\Gamma_{\mathbf{r}}\backslash H_n),\mathcal{O}_0)$ is precompact if and only if there are positive constants $C_1,C_2,C_3>0$ such that all $[\mathbf{A}]\in\mathcal{X}$ satisfies the following three conditions;
		\begin{description}
			\item[(A-1)]$\min\left\{\|X\|_\mathbf{A}\mid X\in \mz_{\mathbf{r}}\right\}\geq C_1$,
			\item[(A-2)]$|\det(\tilde{\mathbf{A}})|\geq C_2$,
			\item[(A-3)]$d_n(\mathbf{A})\leq C_3$,
		\end{description}
	\end{theorem}

 \section{Proof of main propositions}\label{sec5}

In this section,
we show the main propositions.
First we show that the geometric condition in Theorems \ref{mainthm} and \ref{mainthm2} implies finite diffeomorphism types.
Next we show that these Boldt's condition are equivalent to the geometric conditions in Theorem \ref{mainthm} and \ref{mainthm2}.
After stating preparation lemmas,
we pass to the proof for Riemannian and sub-Riemannian case respectively.

\begin{lemma}\label{lemfundamental}
		For any matrix $\mathbf{A}\in\End_\R(\h_n)$,
		we have
		\begin{itemize}
			\item[1]$\delta(\mathbf{A})=\sqrt{2\sum_{i=1}^nd_i(\mathbf{A})^2}$,
			\item[2] $|\det(\tilde{\mathbf{A}})|=\prod_{i=1}^nd_i(\mathbf{A})$.
		\end{itemize}
	\end{lemma}
	\begin{proof}
		The first equality follows from the definition of the canonical form and Lemma \ref{lemdelta}.
		The second equality follows from
		$$\prod_{i=1}^n d_i(\mathbf{A})^2=\det(^{t}\tilde{\mathbf{A}}\mathbf{J}_n\tilde{\mathbf{A}})=\det(\tilde{\mathbf{A}})^2.$$
	\end{proof}

\begin{lemma}[Proposition 3.11 in \cite{ebe} for Riemannian case]\label{lemhorizontal}
	For $U\in\mv\subset\H_n$ and $V\in[\h_n,\h_n]\subset\H_n$,
we have
$$dist_\mathbf{A}(e,U+V\geq\|U\|_\mathbf{A}.$$
Moreover,
the equality holds if and only if $V=0$.
\end{lemma}

\begin{proof}
	
	Let $c$ be a length minimizing path from $e$ to $U+V$.
	Then the composition $\Pr\circ c$ is a path in $\mv$ from $0$ to $U$.
	Moreover the length of $c$ in $(\H_n,\Im(\mathbf{A}),\langle\cdot,\cdot\rangle_\mathbf{A})$ is greater than or equal to that of $\Pr(c)$ in $(\mv,\langle\cdot,\cdot\rangle_\mathbf{A})$,
	thus we obtain the desired inequality.
	
	The equality holds only when
	 \begin{equation}\label{eqlength}
		length(c)=length(\Pr\circ c)=\|U\|_\mathbf{A}.
	 \end{equation}
	 From the uniqueness of geodesics in the scalar product space $(\mv,\langle\cdot,\cdot\rangle_\mathbf{A})$,
	the second equality in (\ref{eqlength}) holds if and only if $V=0$.
	The first equality automatically follows provided $V=0$.
\end{proof}

\subsection{Propositions on Riemannian metrics}\label{secriem2}

\begin{proof}[Proof of Proposition \ref{propfiniteriem}]
Let $(\Gamma_{\mathbf{r}}\backslash \H_n,\langle\cdot,\cdot\rangle_\mathbf{A})$ be a Riemannian compact Heisenberg nilmanifold,
where $\mathbf{r}\in D_n$ and $\mathbf{A}=\begin{pmatrix}
    \tilde{\mathbf{A}} & 0\\
    0 & \rho_\mathbf{A}
\end{pmatrix}\in \End_\R(\h_n)$ is an invertible matrix of canonical form.
We show that if it satisfies the assumption of Theorem \ref{mainthm},
there is $C(n,D,V,K)>0$ such that $\prod_{i=1}^nr_i\leq C(n,D,V,K)$.
Since the fundamental group determines a diffeomorphism type of compact Heisenberg manifolds,
this concludes the proposition.

	From the Riemannian metric $\langle\cdot,\cdot\rangle_\mathbf{A}$ on $\Gamma_{\mathbf{r}}\backslash \H_n$,
	we obtain the quotient flat Riemannian metric on $\mz_{\mathbf{r}}\backslash \mv$.
Since an orthonormal basis of $(\Gamma_{\mathbf{r}}\backslash \H_n,\langle\cdot,\cdot\rangle_\mathbf{A})$ is $\{\mathbf{A}X_1,\dots,\mathbf{A}Y_n,\mathbf{A}Z\}$,
an orthonormal basis of the quotient metric on $\mz_{\mathbf{r}}\backslash \mv$ is $\{\tilde{\mathbf{A}}X_1,\dots,\tilde{\mathbf{A}}Y_n\}$.
This fact allows us to denote the quotient Riemannian metric by $\langle\cdot,\cdot\rangle_{\tilde{\mathbf{A}}}$

The total measure of the torus $(\mz_{\mathbf{r}}\backslash\mv,\langle\cdot,\cdot\rangle_{\tilde{\mathbf{A}}})$ is
$$\left|\int_{\mz_{\mathbf{r}}\backslash \mv}\dv_R(\tilde{\mathbf{A}})\right|=\prod_{i=1}^nr_i|\det(\tilde{\mathbf{A}})|^{-1}.$$

Since the submersion is metric decreasing,
the diameter of the torus $(\mz_{\mathbf{r}}\backslash \mv,\langle\cdot,\cdot\rangle_{\tilde{\mathbf{A}}})$ is less than or equal to $D$.
Therefore we have
	\[\prod_{i=1}^nr_i|\det(\tilde{\mathbf{A}})|^{-1}\leq \omega_{2n}D^{2n},
\]
where $\omega_{2n}D^{2n}$ is the volume of $2n$-dimensional Euclidean ball of radius $D$.
In particular we have
\begin{equation}\label{bounddet}
	\prod_{i=1}^nr_i\leq |\det(\tilde{\mathbf{A}})|\omega_{2n}D^{2n}.
\end{equation}
By the inequality (\ref{bounddet}),
we obtain the upper and lower bound of the total measure of the compact Heisenberg manifold (\ref{measureR}) by 
\begin{align*}
	V\leq meas(\Gamma_{\mathbf{r}}\backslash \H_n,\mathbf{A})=\prod_{i=1}^nr_i|\det(\tilde{\mathbf{A}})|^{-1}|\rho_\mathbf{A}|^{-1}\leq \omega_{2n}D^{2n}|\rho_\mathbf{A}|^{-1}.
\end{align*}
In particular we obtain an upper bound of $|\rho_\mathbf{A}|$ by
\begin{equation}\label{boundrho}
    |\rho_\mathbf{A}|\leq \omega_{2n}D^{2n}/V.
\end{equation}

By Lemma \ref{lemric},
for $i=1,\dots,n$,
we have
\[K\leq \Ric_\mathbf{A}(\mathbf{A}X_i)=-\frac{1}{\rho_\mathbf{A}^2}d_i(\mathbf{A})^2.\]
Combined with (\ref{boundrho}) and Lemma \ref{lemfundamental},
\begin{equation}\label{bounddetabove}
    |\det(\tilde{\mathbf{A}})|=\prod_{i=1}^nd_i(\mathbf{A})\leq |\rho_\mathbf{A}|^n|K|^{\frac{n}{2}}\leq \left(\omega_{2n}D^{2n}/V\right)^{n}|K|^{\frac{n}{2}}.
\end{equation}

By (\ref{bounddet}) and (\ref{bounddetabove}),
we obtain an upper bound of $\prod_{i=1}^nr_i$ by
\[\prod_{i=1}^nr_i\leq \left(\omega_{2n}D^{2n}/V\right)^{n+1}V|K|^{\frac{n}{2}}.\]

\end{proof}

For the proof of Proposition \ref{propric},
we need the eigenvalues of the symmetric positive definite matrix $^{t}\tilde{\mathbf{A}}^{-1}\tilde{\mathbf{A}}^{-1}$.
\begin{definition}
Denote by $\lambda_1(\tilde{\mathbf{A}})\leq \cdots\leq \lambda_{2n}(\tilde{\mathbf{A}})$ the positive eigenvalues of the positive definite symmetric matrix $^{t}\tilde{\mathbf{A}}^{-1}\tilde{\mathbf{A}}^{-1}$.
\end{definition}

\begin{lemma}\label{lema1}
If $diam(\mz_{\mathbf{r}}\backslash\mv,\langle\cdot,\cdot\rangle_{\tilde{\mathbf{A}}})\leq D$ and $\det(\tilde{\mathbf{A}})\leq C$,
then we have
    \[\sqrt{\lambda_1(\tilde{\mathbf{A}})}\geq C^{-1}(2D)^{-n+\frac{1}{2}}.\]
\end{lemma}

\begin{proof}
    
		For $i=1,\dots,n$,
		let $\gamma_i^{\frac{1}{2}}=\exp(\frac{r_i}{2}X_i)$.
		It is easy to check that the length minimizing geodesic from $\Gamma_{\mathbf{r}}e$ to $\Gamma_{\mathbf{r}}\gamma_i^{\frac{1}{2}}$ is the projection of the geodesic in $H_n$ from $e$ to $\gamma_i^{\frac{1}{2}}$.
		By Lemma \ref{lemhorizontal},
		\begin{align*}
			\left\|\frac{\tilde{r}_i}{2}X_i\right\|_\mathbf{A}&=dist_\mathbf{A}\left(e,\gamma_i^{\frac{1}{2}}\right)\\
						      &=\overline{dist}_\mathbf{A}\left(\Gamma_{\mathbf{r}} e,\Gamma_{\mathbf{r}} \gamma_i^{\frac{1}{2}}\right)\\
						      &\leq diam\left(\Gamma_{\mathbf{r}}\backslash H_n,dist_\mathbf{A}\right)\leq D.
		\end{align*}

		Since $r_i\geq1$,
		\[
			\|X_i\|_\mathbf{A}\leq \frac{2D}{r_i}\leq 2D.
		\]

		In the same way,
		we obtain the inequality
		\[
			\|Y_n\|_\mathbf{A}\leq 2D
		\]
		for $i=1,\dots,n$.

By using these inequalities,
for $X=\sum_{i=1}a_iX_i+b_iY_i\in\mv$,
we have
\[\langle X,^{t}\tilde{\mathbf{A}}^{-1}\tilde{\mathbf{A}}^{-1}X\rangle_0=\|X\|_\mathbf{A}\leq \sum_{i=1}^n|a_i|\|X_i\|_\mathbf{A}+|b_i|\|Y_i\|_\mathbf{A}\leq 2D\sum_{i=1}^n|a_i|+|b_i|\leq 2D\|X\|_0.\]
Therefore the maximal eigenvalue $\lambda_{2n}(\tilde{\mathbf{A}})\leq 2D$.

On the other hand,
since $\det(\tilde{\mathbf{A}})\leq C$,
we have
\[C^{-1}\leq \det(\tilde{\mathbf{A}}^{-1})=\left(\prod_{i=1}^{2n}\lambda_i(\tilde{\mathbf{A}})\right)^{\frac{1}{2}}\leq \sqrt{\lambda_1(\tilde{\mathbf{A}})(2D)^{2n-1}}.\]
Therefore we obtain
\[\sqrt{\lambda_1(\tilde{\mathbf{A}})}\geq C^{-1}(2D)^{-n+\frac{1}{2}}.\]
\end{proof}

	\begin{proof}[Proof of Proposition \ref{propriem}]
 Fix $\mathbf{r}\in D_n$.
 We show that if a compact Riemannian Heisenberg manifold $(\Gamma_{\mathbf{r}}\backslash \H_n,\langle\cdot,\cdot\rangle_\mathbf{A})$ satisfies the following three conditions,
 then there are constants $C_1,C_2,C_3,C_{\pm}$ depending only on $D,V,K$ such that the conditions (A-1)-(A-4) hold;
 \begin{itemize}
     \item[1] $diam(\Gamma_{\mathbf{r}}\backslash \H_n,\langle\cdot,\cdot\rangle_\mathbf{A})\leq D$
     \item[2] $meas(\Gamma_{\mathbf{r}}\backslash \H_n,\langle\cdot,\cdot\rangle_\mathbf{A})=\prod_{i=1}^nr_i\left|\det(\tilde{\mathbf{A}})\rho_\mathbf{A}\right|^{-1}\geq V$,
     \item[3] $\Ric_\mathbf{A}\geq K$.
 \end{itemize}
 The proof of the converse implication duplicates with the proof of Theorem \ref{thmboltext} by Boldt \cite{bol2},
 so we omit.

First we show the inequality (A-1) with the constant $C_1=\frac{V^n}{\omega_{2n}^n|K|^{\frac{n}{2}}2^{n-\frac{1}{2}}D^{2n^2+n-\frac{1}{2}}}$.
From the definition of the quotient metric $\langle\cdot,\cdot\rangle_{\tilde{\mathbf{A}}}$ and the lattice $\mz_{\mathbf{r}}$,
we have

\begin{align*}
&\min\left\{\|X\|_\mathbf{A}\mid X\in \mz_{\mathbf{r}}\right\}\\
\geq & \min\left\{\sqrt{\langle X,X\rangle_{\tilde{\mathbf{A}}}}\mid \|X\|_0\leq 1,X\in\mv\right\}\\
= & \min\left\{\sqrt{\langle X,^{t}\tilde{\mathbf{A}}^{-1}\tilde{\mathbf{A}}^{-1}X\rangle_0}\mid \|X\|_0\leq 1,X\in \mv\right\}\\
= & \sqrt{\lambda_1(\tilde{\mathbf{A}})}\\
\geq & \frac{V^n}{\omega_{2n}^n|K|^{\frac{n}{2}}2^{n-\frac{1}{2}}D^{2n^2+n-\frac{1}{2}}}.
\end{align*}
For the last inequality we use (\ref{bounddetabove}) and Lemma \ref{lema1}.

The inequality (A-2) holds with the constant $C_2=\frac{\prod_{i=1}^nr_i}{\omega_{2n}D^n}$ by (\ref{bounddet}).
Here notice that $\prod_{i=1}^nr_i$ is fixed.

		Next we show the inequality (A-3) with the constant $C_3=\frac{\sqrt{2|K|}\omega_{2n}D^{2n}}{V}$.
  By Lemma \ref{lemric} and the lower bound of the Ricci curvature,
\begin{equation}\label{eqa-3}
	Ric_\mathbf{A}(\mathbf{A}X_n)=-\frac{1}{2\rho_\mathbf{A}^2}d_n(\mathbf{A})^2\geq K.
\end{equation}

Combined with (\ref{boundrho}),
we obtain
$$d_n(\mathbf{A})\leq \sqrt{2|K|}|\rho_\mathbf{A}|\leq \frac{\sqrt{2|K|}\omega_{2n}D^{2n}}{V}.$$

 Next we show the inequality (A-4).
 The upper bound is $C_+=\frac{\omega_{2n}D^{2n}}{V}$ by (\ref{boundrho}).
 The lower bound is given by $C_-=\frac{\sqrt[n]{C_2}}{\sqrt{2K}}$,
 where $C_2$ is the constant given above.
Indeed,
combined with Lemma \ref{lemric} and \ref{lemfundamental},
we obtain
$$|\rho_\mathbf{A}|\geq \frac{1}{\sqrt{2|K|}}d_n(\mathbf{A})\geq \frac{1}{\sqrt{2|K|}}\sqrt[n]{\left|\det(\tilde{\mathbf{A}})\right|}\geq \frac{\sqrt[n]{C_2}}{\sqrt{2|K|}}.$$

	\end{proof}

\subsection{Propositions on sub-Riemannian metrics}

\begin{proof}[Proof of Proposition \ref{propfinitesR}]
    
Let $(\Gamma_{\mathbf{r}}\backslash \H_n,\langle\cdot,\cdot\rangle_\mathbf{A})$ be a sub-Riemannian compact Heisenberg nilmanifold,
where $\mathbf{r}\in D_n$ and $\mathbf{A}=\begin{pmatrix}
    \tilde{\mathbf{A}} & 0\\
    0 & 0
\end{pmatrix}\in \End_\R(\h_n)$ is a corank $1$ matrix of canonical form.
We show that if it satisfies the assumption of Theorem \ref{mainthm2},
there is $C(n,D,V)>0$ such that $\prod_{i=1}^nr_i\leq C(n,D,V)$.

In the same way with Riemannian case,
	we obtain the quotient flat Riemannian metric $\langle\cdot,\cdot\rangle_{\tilde{\mathbf{A}}}$ on $\mz_{\mathbf{r}}\backslash \mv$ whose orthonormal basis is $\{\tilde{\mathbf{A}}X_1,\dots,\tilde{\mathbf{A}}Y_n\}$.
Moreover we obtain the inequality (\ref{bounddet}),
that is
\[\prod_{i=1}^nr_i\leq |\det(\tilde{\mathbf{A}})|\omega_{2n}D^{2n}.\]
We shall give an upper bound of $|\det(\tilde{\mathbf{A}})|$.
By the above inequality and the assumption,
we obtain the upper and lower bound of the total measure (\ref{measuresR}) by 
\begin{align*}
	V\leq meas(\Gamma_{\mathbf{r}}\backslash \H_n,\mathbf{A})=\prod_{i=1}^nr_i|\det(\tilde{\mathbf{A}})|^{-1}\delta(\mathbf{A})^{-1}\leq \omega_{2n}D^{2n}\delta(\mathbf{A})^{-1}.
\end{align*}
In particular we obtain an upper bound of $|\delta_\mathbf{A}|$ by
\begin{equation}\label{bounddelta}
    |\delta(\mathbf{A})|\leq \omega_{2n}D^{2n}/V.
\end{equation}

By Lemma \ref{lemfundamental},
\[
    |\det(\tilde{\mathbf{A}})|=\prod_{i=1}^nd_i(\mathbf{A})\leq d_n(\mathbf{A})^n\leq \left(\sqrt{2\sum_{i=1}^n d_i(\mathbf{A})^2}\right)^n=\delta(\mathbf{A})^n.
\]
Therefore we obtain the upper bound of $\det(\tilde{\mathbf{A}})$ by
\begin{equation}\label{bounddet2}
    |\det(\tilde{\mathbf{A}})|\leq \omega_{2n}^nD^{2n^2}/V^n
\end{equation}
By (\ref{bounddet}) and (\ref{bounddet2}),
we obtain
\[\prod_{i=1}^nr_i\leq \omega_{2n}^{1+\frac{1}{n}}D^{2n+2}/V^{\frac{1}{n}}.\]

\end{proof}

\begin{proof}[Proof of Proposition \ref{propsR}]
	(A-1) holds with $C_1=V^n/(\omega_{2n}^{n}D^{n^2}(2D)^{n-\frac{1}{2}})$ by Lemma \ref{lema1} and the upper bound of $\det(\tilde{\mathbf{A}})$ in (\ref{bounddet2}).
 
 (A-2) holds with $C_2=\prod_{i=1}^nr_i/(\omega_{2n}D^{2n})$ by (\ref{bounddet}).

 (A-3) holds with $C_3=\omega_{2n}D^{2n}/V$ by (\ref{bounddelta}) and the trivial inequality $d_n(\mathbf{A})\leq \delta(\mathbf{A})$.
 \end{proof}

\end{document}